\documentclass{tsp-short}

\theoremstyle{plain}
\newtheorem{thm}{Theorem}[section]
\newtheorem{lem}{Lemma}[section]

\newtheorem{corl}{Corollary}[section]
\theoremstyle{definition}

\theoremstyle{remark}

\title[On existence and properties of strong solutions ]{On existence and properties of strong solutions of one-dimensional stochastic equations with an additive noise}
\author{Andrey Yu. Pilipenko}
\thanks{Research is partially supported by State fund for fundamental researches of Ukraine and the Russian foundation for basic researches
    Grant F40.1/023}
\thanks{Research is partially supported by the Grant of the President of Ukraine,
    Grant F47/033}

\address{Institute of Mathematics,  National Academy of Sciences of
 Ukraine, Tereshchenkivska str. 3, 01601, Kiev, Ukraine
}
\email{apilip@imath.kiev.ua}


\begin{document}

\begin{abstract}
One-dimensional stochastic differential equations with additive L\'evy noise are considered. Conditions for existence and uniqueness of a strong solution  are obtained. In particular, if the noise is a L\'evy symmetric stable process with $\alpha\in(1;2)$, then the measurability and boundedness of a drift term is sufficient for the existence of a strong solution. We also study continuous dependence  of the strong solution on the initial value and the drift.
\end{abstract}

\maketitle

\section*{Introduction}
Consider an SDE
\begin{equation}
\label{eq1.1}
\xi(t)=x+\int^t_0a(\xi(s))ds+Z(t), \ t\geq0,
\end{equation}
where $a: {\mathbb R}\to{\mathbb R}$ is a measurable function, $Z$ is a L\'evy process.  We study a question of
 existence and uniqueness for   the strong solution of \eqref{eq1.1}, and also its continuous
 dependence on initial value $x$ and a function $a.$

At first we obtain a few general results and then apply them to the case, where $Z$ is a symmetric stable process with $\alpha\in(1,2).$
In particular, in this case the strong solution exists and is unique if $a$ is bounded. Moreover,
let $\{\xi_n, n\geq1\}$ be a sequence that satisfies \eqref{eq1.1} with initial values
$\{x_n, n\geq1\}$ and drift functions $\{a_n, n\geq1\}. $ We prove that if $x_n$
converges to $x,$ $a_n$ converges to $a$ almost surely with respect to the Lebesgue measure, and
 a sequence of functions $\{a_n, n\geq1\}$ is uniformly bounded, then we have the
 uniform convergence of solutions in probability:
$$
\forall \ T>0: \ \sup_{t\in[0, T]}|\xi_n(t)-\xi(t)|\overset{P}{\rightarrow}0, \ n\to\infty.
$$
A lot of ideas and methods of investigation are quite standard. We use the
Yamada--Watanabe theorem, we prove that the minimum of two solutions is a
solution, we use the Skorokhod's method of a common probability space. However we
cannot find in the literature the direct reference to a general result which
can be applied to SDEs with L\'evy noise.

\section{Pathwise uniqueness}
\label{section1}

In this section we prove that a weak uniqueness of \eqref{eq1.1} yields a pathwise uniqueness.
If we suppose also  existence of a weak solution, then reasoning of the Yamada--Watanabe theorem and some
minor technical assumptions will imply  existence and uniqueness of the strong solution.

Let $a: {\mathbb R}\to{\mathbb R}, \ Z: [0, \infty)\to{\mathbb R}$ be measurable (non-random) functions.
Consider the equation
$$
\xi(t)=x+\int^t_0a(\xi(s))ds+Z(t), \ t\geq0.
$$
We will assume by definition that if $\xi$ is a solution of this integral equation, then
$$
\int^T_0|a(\xi(s))|ds<\infty,
$$
for any $T>0.$

We need the following simple statement about solutions of non-random integral equations.

\begin{lem}
\label{lem2.1}
 Assume that
 measurable functions $\xi_i: [0, \infty)\to{\mathbb R}, i=1,2,$ satisfy the equation
\begin{equation}
\label{eq3.1}
\xi_i(t)=x+\int^t_0a(\xi_i(s))ds+Z(t), \ t\geq0.
\end{equation}
Then
$
\xi_-(t)=\xi_1(t)\wedge\xi_2(t)
$ and $\xi_+(t)=\xi_1(t)\vee\xi_2(t)$ are also solutions of \eqref{eq3.1}.
\end{lem}

\begin{proof}
At first let us observe that
$$
\int^T_0|a(\xi_\pm(s))|ds\leq\int^T_0(|a(\xi_1(s))|+|a(\xi_2(s))|)ds<\infty,
$$
so integrals $\int^T_0a(\xi_\pm(s))ds$ are well-defined.

Let us show that $\xi(t)=\xi_-(t)$ is a solution of \eqref{eq3.1}. The reasoning for $
\xi_+(t)$ is the same. Since the function $\xi_1(t)-\xi_2(t)=\int^t_0(a(\xi_1(s))-a(\xi_2(s)))
ds$ is continuous, the set
$$
U=\{t\geq0: \xi_1(t)\ne\xi_2(t)\}
$$
is open.

Let $U=\cup_k(b_k, c_k),$ where $(b_k, c_k)\cap(b_j, c_j)=\O$ for $k\ne j$ (possibly
$c_k=\infty$ for some $k$).

For any $k$ the only one of equalities is satisfied, either $\xi(t)=\xi_1(t), t\in(b_k, c_k),$ or
 $\xi(t)=\xi_2(t), t\in(b_k, c_k).$ Moreover, if $c_k\ne\infty,$ then
$$
 \xi_1(b_k)=\xi_2(b_k), \ \xi_1(c_k)=\xi_2(c_k).
$$
This yields
\begin{equation}
\label{eq4.1}
\begin{split}
\int^{c_k}_{b_k}a(\xi_1(s))ds=\int^{c_k}_{b_k}a(\xi_2(s))ds=\int^{c_k}_{b_k}a(\xi(s))ds=\\
=-Z(c_k)+\xi_1(c_k)+Z(b_k)-\xi_1(b_k).
\end{split}
\end{equation}
Let $t\in(b_n, c_n).$ Assume that $\xi_1(t)<\xi_2(t).$ Then
$$
\int^t_0a(\xi(s))ds=\Bigg(\int_{[0,t]\setminus U}+\sum_{(b_k, c_k)\subset[0, t]}\int^{c_k}_{b_k}
+\int^t_{b_n}\Bigg)
a(\xi(s))ds.
$$
For any $s\notin U: \xi(s)=\xi_1(s)=\xi_2(s).$ So the first integral equals
$$
\int_{[0,t]\setminus U} a(\xi_1((s))ds.
$$
Due to \eqref{eq4.1} we have that the second integral is equal to
$$
\sum_{(b_k, c_k)\subset[0,t)}\int^{c_k}_{b_k}a(\xi_1(s))ds.
$$

For any $s\in(b_n, c_n): \xi_1(s)<\xi_2(s).$ So
$$
\xi(s)=\xi_1(s)\wedge\xi_2(s)=\xi_1(s), \ s\in(b_n, t).
$$
Thus the third integral equals
$$
\int^t_{b_n} a(\xi_1(s))ds.
$$
That is
$$
\int^t_0a(\xi(s))ds=\int^t_0a(\xi_1(s))ds,
$$
$$
\xi(t)=\xi_1(t)=x+\int^t_0a(\xi_1(s))ds+Z(t)=x+\int^t_0a(\xi(s))ds+Z(t).
$$
The case $t\notin U$ can be considered analogously. Lemma \ref{lem2.1} is
proved.
\end{proof}

Let now $Z(t), t\geq0,$ be a L\'evy  process defined on a filtered probability space $(\Omega, {\mathcal F}, {\mathcal F}_t, P)$. In this case we will consider only $({\mathcal F}_t)$-adapted solutions of
 \eqref{eq1.1}.

Lemma \ref{lem2.1} and weak uniqueness of solution of \eqref{eq1.1} imply pathwise uniqueness.
For the corresponding definitions see for example \cite{RevuzYor}, Ch.IX \S\, 1.

\begin{corl}
\label{corl1}
Assume that \eqref{eq1.1} satisfies the weak uniqueness property. Then we have the pathwise uniqueness
for a solutions of \eqref{eq1.1}.
\end{corl}
Really, let $\xi_1(t)$ and $\xi_2(t)$ be solutions of
\eqref{eq1.1} defined on the same filtered probability space. Then $\xi_-(t)=\xi_1(t)\wedge\xi_2(t)$ and $\xi_+(t)=\xi_1(t)\vee\xi_2(t)$  are
also solutions of \eqref{eq1.1}. Trajectories of $\xi_1$ and $\xi_2$ are c\'adl\'ag. So, if
$$
P(\exists \ t\geq0: \xi_1(t)\ne \xi_2(t))>0,
$$
then
$$
\exists \ t\geq0: \ P(\xi_1(t)\ne\xi_2(t))>0,
$$
and hence
$$
\exists \ t\geq0:   \ P(\xi_-(t)<\xi_+(t))>0.
$$
Since $\xi_-(t)\leq\xi_+(t),$ the distributions of random variables $\xi_-(t)$ and $\xi_+(t)$ cannot coincide.
This contradicts weak uniqueness. Thus
$$
P(\forall \ t\geq0: \xi_1(t)=\xi_2(t))=1.
$$

Applying the Yamada--Watanabe theorem and Corollary \ref{corl1} we obtain the following
statement on existence of the strong solution (the formulation of the Yamada--Watanabe theorem was given for Wiener noise,
 but the proof can be applied to our situation almost without changes).

\begin{corl}
\label{corl2}
Assume that there exists a unique weak solution of \eqref{eq1.1}. Then this solution is a strong solution.
%there exists a unique strong solution of \eqref{eq1.1}.
\end{corl}
As an application of Corollary \ref{corl2} let us consider the case when $Z(t), t\geq0,$ is a
symmetric stable process, i.e. $Z(t), t\geq0,$ is a c\'adl\'ag process with stationary
independent increments and
$$
\exists  \alpha\in(0, 2]\ \exists c>0 \ \forall  \lambda\in{\mathbb R}\ \forall  t\geq0: \
E\exp\{i\lambda Z(t)\}=\exp\{-ct|\lambda|^\alpha\}.
$$

We need the following result on existence and uniqueness, and properties of weak solution of
\eqref{eq1.1} with symmetric stable noise.

\begin{thm}
\label{thm1.1}
Assume that $Z(t), t\geq0,$ is a symmetric stable process with $\alpha\in(1, 2).$

1) If $a\in L_\infty({\mathbb R}),$ then there exists a unique weak solution to \eqref{eq1.1}.

2) If $a\in L_p({\mathbb R})$ for some $p\in\Bigg(\frac{1}{\alpha-1}; +\infty\Bigg],$ then there exists a weak
solution of \eqref{eq1.1} such that

a) $\xi$ is a Markov process with a continuous transition probability density $p(t,x,y), t>0,\ x\in{\mathbb R},\ y\in{\mathbb R};$

b) for any $T>0$ there exists a constant $N=N{(T, \|a\|_{L_p})}$ such that
$$
\forall \ t\in(0, T] \ \forall  x,y\in{\mathbb R} \ \forall  \ k\in\{0;1\}: \
\Bigg|\frac{\partial^kp(t,x,y)}{\partial x^k}\Bigg|\leq\frac{Nt}{(t+|x-y|)^{\alpha+k+1}}.
$$
\end{thm}

For the proof of the first item see \cite{Komatsu_stable1982}, the second one see in \cite{Portenko_Levy, PortenkoPodolynny}.

\begin{corl}
\label{corl3}
Let $Z(t), t\geq0,$ be a symmetric stable process with $\alpha\in(1, 2)$ and $a\in L_\infty({\mathbb R}).$ Then there exists a unique strong solution to \eqref{eq1.1}.
\end{corl}

{\it Remark}. Using a localization technique it is not difficult to prove
the existence of a unique solution to \eqref{eq1.1} if a measurable function $a$ satisfies
a linear growth condition:
$$
\exists K\ \forall x:\ \ |a(x)|\leq K(1+|x|).
$$

\section{Continuous dependence on initial conditions and coefficients of the equation}
\label{section2}

Assume that $\{\xi_n(t), t\geq0\}, n\geq0,$ are solutions of the equations
\begin{equation}
\label{eq13.1}
\xi_n(t)=x_n+\int^t_0a_n(\xi_n(s))ds+Z(t), t\geq0,
\end{equation}
where $\{Z(t), t\geq0\}$ is  a L\'evy  process defined on a filtered probability space $(\Omega, {\mathcal F}, {\mathcal F}_t, P)$.
As in the previous Section we
also require
  ${\mathcal F}_t$-measurability of $\xi_n(t).$

The main result of this Section is the Theorem and Corollary below.
\begin{thm}
\label{thm2.1}
Assume that

1) $\lim_{n\to\infty}x_n=x_0;$

2) $\sup_{n\geq0}\sup_x|a_n(x)|<\infty;$

3) there exists a finite measure $\mu$ on ${\mathcal B}({\mathbb R})$ such that for any $n\geq1$ and
$\lambda$-a.a. $t\geq0$ ($\lambda$ is the Lebesgue measure) the distribution of $\xi_n(t)$ has a density $p_n(x, t)$ w.r.t.
$\mu(dx);$

4) $a_n\to a_0, n\to\infty,$ in measure $\mu;$

5) for $\lambda$-a.a. $t\geq0$  a sequence $\{p_n(\cdot, t), n\geq1\}$ is
uniformly integrable w.r.t. $\mu;$

6) there exists a unique solution to equation \eqref{eq13.1} where $n=0.$

Then for any $T>0:$
\begin{equation}
\label{eq14.1}
\sup_{t\in[0, T]}|\xi_n(t)-\xi_0(t)|\overset{P}{\rightarrow}0, \ n\to\infty.
\end{equation}
\end{thm}

This theorem, Corollary \ref{corl2} and Theorem \ref{thm1.1} imply the following result on the
continuous dependence on a parameter for the  solution of \eqref{eq13.1} with a stable noise.

\begin{corl}
\label{corl4}
Let $\{Z(t), t\geq0\}$ be a symmetric stable process with $\alpha\in(1, 2).$ Assume that items
1), 2) and 4) of Theorem \ref{thm2.1} are satisfied, where $\mu(dx)=(1+|x|)^{\alpha+1}dx$. Then \eqref{eq13.1} has a unique strong
solution for any $n\geq0$ and \eqref{eq14.1} holds true.
\end{corl}
{\it Remark}.  The convergence of a sequence of functions in the measure $\mu$ is equivalent the convergence in any absolute continuous finite measure with positive density.

\begin{proof}[Proof of Theorem \ref{thm2.1}] We use the Skorokhod idea of using a common
probability space \cite{Skor-issl}, Ch.1 \S 6, Ch.3 \S 3. Consider a sequence of processes
$$
X_n(\cdot)=(\xi_n(\cdot), \xi_0(\cdot),  Z(\cdot), x_n+\int^\cdot_0a_n(\xi_n(s))ds, x_0+\int^\cdot_0a_
0(\xi_0(s))ds), n\geq1
$$
as a sequence with values in
$$
(D([0, T]))^3\times(C([0, T]))^2.
$$
It easily follows from the assumptions 1), 2) of the Theorem that this sequence is tight. So,
 there exists a weakly convergent subsequence $\{X_{n_k}\}.$ Without loss of generality we will
 assume that $\{X_n\}$ is weakly convergent itself.

By the Skorokhod theorem \cite{Skor-issl}, Ch.1 \S 6, there exists a new probability
space and a sequence $\{\widetilde{X}_n, n\geq1\}$ such that
$\widetilde{X}_n\overset{d}{=}X_n, n\geq1,$ and $\{\widetilde{X}_n, n\geq1\}$ converges in
probability to some random element $\widetilde{X}_0.$ Denote the three first
coordinates of $\{\widetilde{X}_n, n\geq1\}$ by $\widetilde{\xi}_n(\cdot), \widehat{\xi}_n(\cdot),
Z_n(\cdot).$ Note that the fourth and the fifth coordinates of $\{\widetilde{X}_n,
n\geq1\}$ are measurable functions of the first and the second one. So they are
equal to $x_n+\int^\cdot_0a_n(\widetilde{\xi}_n(s))ds,$
$x_0+\int^\cdot_0a_0(\widehat{\xi}_n(s))ds,$ respectively.

Let
$$\widetilde{X}_0=(\widetilde{\xi}_0(\cdot), \widehat{\xi}_0(\cdot), Z_0(\cdot), \alpha(\cdot), \beta(\cdot)),$$
where  $\alpha(t), \beta(t), t\in[0, T],$ are continuous processes. We have not known yet  that
$$
\alpha(t)=x_0+\int^t_0a_0(\widetilde{\xi}_0(s))ds, \ \beta(t)=x_0+\int^t_0a_0(\widehat{\xi}_0(s))ds.
$$
Note  that for any $t\in[0, T]$ random variables $\widetilde{\xi}_0(t)$ and $\widehat{\xi}_0(t)$ are
 independent of $\sigma(Z_0(t+s)-Z_0(t), s\geq0).$

Let  us verify that $\widetilde{\xi}_0$ is a  solution of the equation
\begin{equation}
\label{eq16.2}
\widetilde{\xi}_0(t)=x_0+\int^t_0a_0(\widetilde{\xi}_0(s))ds+Z_0(t), t\in[0, T].
\end{equation}
To prove this it is sufficient to prove that for $\lambda$-a.a. $t\in[0, T]:$
\begin{equation}
\label{eq16.0}
x_0+\int^t_0a_0(\widetilde{\xi}_0(s))ds=\alpha(t)  \ \mbox{a.s.}
\end{equation}
It follows from the convergence in probability in $D([0, T])$ that for all $t\in[0, T],$ except
 of possibly countable set, a convergence in probability
\begin{equation}
\label{eq16.1}
\widetilde{\xi}_n(t)\overset{P}{\rightarrow}\widetilde{\xi}_0(t)
\end{equation}
holds.

\begin{lem}
\label{lem17}
Let $\{\eta_n, n\geq0\}$ be a sequence of random variables. Assume that for any $n\geq1$ the
 distribution of $\eta_n$ is absolutely continuous w.r.t. a probability measure $\mu.$ Denote the
  corresponding density by $p_n.$  Let $\{a_n, n\geq0\}$ be a sequence of measurable functions on ${\mathbb R}$.
  Suppose that the following conditions are satisfied:

1) $\eta_n\overset{P}{\rightarrow}\eta_0, \ n\to\infty;$

2) $a_n\overset{\mu}{\rightarrow}a_0, \ n\to\infty;$

3) a sequence of densities $\{p_n, n\geq1\}$ is uniformly integrable w.r.t. $\mu.$

Then
$$
a_n(\eta_n)\overset{P}{\rightarrow}a_0(\eta_0), \ n\to\infty.
$$
\end{lem}
The proof  is similar to \cite{KP}, Lemma 2, where it was considered a sequence of random elements with values in a Polish
space. Note that all functions $\{a_n\}$ may be discontinuous.

It follows from Lemma \ref{lem17}, assumptions of the Theorem and \eqref{eq16.1} that for
$\lambda$-a.a.
$s\in[0, T]:$
$$
a_n(\widetilde{\xi}_n(s))\overset{P}{\rightarrow}a_0(\widetilde{\xi}_0(s)), \ n\to\infty.
$$
So
\begin{equation}
\label{eq18.1}
\begin{split}
&E\sup_{t\in[0, T]}|\int^t_0a_n(\widetilde{\xi}_n(s))ds-\int^t_0a_0(\widetilde{\xi}_0(s))ds|\leq\\
&\leq
E\int^t_0|a_n(\widetilde{\xi}_n(s))-a_0(\widetilde{\xi}_0(s))|ds\to0, n\to\infty,
\end{split}
\end{equation}
by Lebesgue theorem on dominated convergence. Thus \eqref{eq16.0} is satisfied and hence
$\widetilde{\xi}_0$ is a solution of \eqref{eq16.2}.

Similarly it can be proved that $\widehat{\xi}_0$ satisfies the same equation
$$
\widehat{\xi}_0(t)=x_0+\int^t_0a_0(\widehat{\xi}_0(s))ds+Z_0(t), t\in[0, T], \ \mbox{a.s.}
$$
Since this equation has a unique solution, we have equality
$$
\widehat{\xi}_0(t)=\widetilde{\xi}_0(t), t\in[0, T], \ \mbox{a.s.}
$$
Let us return to the initial probability space. Let $\varepsilon>0$ be fixed. Then
$$
P(\sup_{t\in[0, T]}|\xi_n(t)-\xi_0(t)|>\varepsilon)=
$$
$$
=P(\sup_{t\in[0, T]}|\widetilde{\xi}_n(t)-\widehat{\xi}_n(t)|>\varepsilon)=
$$
$$
=
P(|x_n-x_0|+\sup_{t\in[0, T]}|\int^t_0a_n(\widetilde{\xi}_n(s))ds-\int^t_0a_n(\widehat{\xi}_n(s))ds|>\varepsilon)
\leq
$$
$$
\leq
P(|x_n-x_0|+\sup_{t\in[0, T]}|\int^t_0a_n(\widetilde{\xi}_n(s))ds-\int^t_0a_0(\widetilde{\xi}_0(s))ds|>
\frac{\varepsilon}{2})+
$$
$$
+
P(\sup_{t\in[0, T]}|\int^t_0a_n(\widehat{\xi}_n(s))ds-\int^t_0a_0(\widehat{\xi}_0(s))ds|>\frac{\varepsilon}{2}).
$$
The items in the r.h.s. converge to zero by \eqref{eq18.1} and similar statement for
$\widehat{\xi}_n.$ The theorem is proved.
\end{proof}

\bibliographystyle{plain}

\end{document}